\documentclass[12 pt]{article}
\usepackage{mathrsfs}
\usepackage{amsfonts}
\usepackage{amssymb}
\usepackage{amsmath}
\usepackage{amsthm}
\usepackage{fullpage}
\usepackage{graphicx}
\usepackage{wrapfig}
\usepackage{caption}

\newtheorem{theorem}{Theorem}[section]
\newtheorem{lemma}{Lemma}[section]
\newtheorem{proposition}{Proposition}[section]
\newtheorem{corollary}{Corollary}[section]
\newtheorem{assumption}{Assumption}[section]
\newtheorem{remark}{Remark}[section]

\title{Finite Horizon Time Inhomogeneous Singular Control Problem of One-dimensional Diffusion via  Dynkin Game}
\author{Yipeng Yang\thanks{Department of Mathematics, University of Houston - Clear Lake (yangy@uhcl.edu)}}

\begin{document}
\maketitle
\begin{abstract} The Hamilton-Jacobi-Bellman equation (HJB)
associated with the time inhomogeneous singular control problem is
 a parabolic partial differential equation, and the existence
of a classical solution is usually difficult to prove. In this
paper, a class of finite horizon stochastic singular control
problems of one dimensional diffusion is solved via a time
inhomogeneous zero-sum game (Dynkin game). The regularity of the
value function of the Dynkin game is investigated, and its
integrated form coincides with the value function of the singular
control problem. We provide conditions under which a classical
solution to the associated HJB equation exists, thus the usual
viscosity solution approach is avoided. We also show that the
optimal control policy is to reflect the diffusion between two time
inhomogeneous boundaries. For a more general terminal payoff
function, we showed that the optimal control involves a possible
impulse at maturity.
\end{abstract}

\


\section{Introduction}
The stochastic singular control problem is one of the classical
research topics in control theory that keeps receiving a lot of
interest in years. A typical such problem has an objective function
as a functional of an underlying stochastic process over a finite or
infinite horizon that needs to be minimized or maximized. This
objective function often involves a functional of the control
actions, which gives its name singular control or impulsive control
depending on the form of the objective function. Its well known
deterministic counterpart is the minimum fuel control problem. The
application of the stochastic singular control is certainly broad,
such as financial engineering, resource management, or mechanical
system control, see, e.g., \cite{Pham09}.

The goal of the stochastic singular control is to characterize the
optimal control policy, if exists, and find the optimal value of the
objective function. In a typical setting, the value function is
known to satisfy a partial differential equation, referred to as the
Hamilton-Jacobi-Bellman (HJB) equation. However, the existence and
regularity of its solution always remain a big challenge. Even in
the time homogeneous one dimensional case, efforts are needed and
special forms are treated in order to characterize the regularity of
the optimal value function, see, e.g., \cite{ma92}\cite{Pham07}. The
 form of optimal policies in  multi-dimensional or time
inhomogeneous singular control problems could be much more
complicated. Soner and Shreve \cite{Soner89} studied a two
dimensional singular stochastic control problem and showed the
existence of a smooth solution to the associated HJB equation, where
the underlying process is a two dimensional Brownian motion with no
drift. How to extend this method to a general multi-dimensional
diffusion is still not clear.

Besides the viscosity solution technique on the study of HJB
equations and stochastic control problems,
e.g.,\cite{Crandall92}\cite{Fleming06}, which often yields a less
regular solution, a lot of literature tackle the singular stochastic
control problem through the optimal stopping and game theory. The
standard optimal stopping problem is often referred to as the one
obstacle problem, and double obstacles problem is referred to as a
game. The connection between singular control  and optimal stopping
is a well known fact. To name a few, Karatzas and Shreve
\cite{Karatzas85} studied the connection between optimal stopping
and singular stochastic control of one dimensional Brownian motion,
and showed that the region of inaction in the control problem is the
optimal continuation region for the stopping problem. Baldursson and
 Karatzas \cite{Baldursson97} established and
exploited the duality between the myopic investor's problem (optimal
stopping) and the social planning problem (stochastic singular
control), where an integral form and change of variable formula were
also presented on this connection. Ma \cite{ma92} dealt with a one
dimensional stochastic singular control problem where the drift term
is assumed to be linear and the diffusion term is assumed to be
smooth, and he showed that the value function is convex and $C^2$
and the controlled process is a reflected diffusion over an
interval. Guo and Tomecek \cite{Guo09} solved a one dimensional
singular control problem via a switching problem, and showed, using
the smooth fit property \cite{Pham07}, that under some conditions
the value function is continuously differentiable ($C^1$). This
connection in a finite horizon case can also be found in
\cite{Boetius98}, where the regularity of the value function is not
fully investigated.

To be brief, there exists a double obstacle problem (game)
associated with a singular control problem, and the derivative (with
respect to state variable) of the value function of the singular
control problem is just the value of this game. The optimal
continuation region of this game coincides with the continuation
region of the singular control problem, and the boundary of the
continuation region turns out to be the reflecting boundary in the
optimal singular control so that the controlled process is a
reflected diffusion within this boundary. Therefore, it is possible
to pass on the regularity of the value function of the game to the
regularity of the value function of the singular control, with extra
smoothness add-on. For example, Fukushima and Taksar \cite{Fuku02}
used this idea to show the existence of a classical solution to the
HJB equation associated with a general one dimensional singular
control problem. See also \cite{Yang14} for an extension and a
correction. Therefore, the regularity of the value function of
optimal stopping and game becomes critical for the study of the
singular stochastic control.

There are several major approaches to the study of optimal stopping
and game, see \cite{Mertens73}\cite{Karoui82} for a general
mathematical framework, and \cite{Zabczyk84} for an early survey. On
the study of regularity of the value function, method of convex
analysis can be found in \cite{Bismut81}, time-discretization method
can be found in \cite{Mackevicius73}, and the penalty method was
introduced by Stettner and Zabczyk \cite{Stettner81}. Another major
approach is via the variational inequalities pioneered by Bensoussan
and Lions \cite{Bensoussan78}. Since the variational inequalities
involve Dirichlet form, it sparks the research interest in the study
of Dirichlet form and its connection to Markov processes, see, e.g.,
\cite{maz92}\cite{Fuku11}. The application of Dirichlet form to
optimal stopping was studied in \cite{Nagai78}. Zabczyk \cite{Zab84}
extended this result to a zero-sum game. Fukushima and Menda
\cite{Fuku06} later investigated the refined solutions under
absolute continuity condition on the transition function of the
underlying process.

However, most of these work dealt with the time homogeneous case. In
this case, the associated HJB equation is an elliptic PDE, and
theories guarantee the existence of a smooth solution under some
conditions, see, e.g., \cite{Fuku02}. For the time inhomogeneous
optimal stopping and game, however, the associated HJB equations are
parabolic PDEs, and the existence and regularity of the solutions
are harder to analyze. For example,  Oshima \cite{Oshima06} applied
the time inhomogeneous Dirichlet form to this problem, and showed
the fine and cofine continuity of the value function, with possible
existence of exceptional sets in the result. Refined solutions to
time inhomogeneous optimal stopping and game via Dirichlet form was
studied in \cite{Ya04}. Palczewski and Stettner
\cite{Palczewski10}\cite{Palczewski11} used the penalty method to
characterize the continuity of the value functions of a time
inhomogeneous optimal stopping problem under weak Feller condition
on the underlying process. Stettner \cite{Stettner11} then extended
this method to finite horizon double obstacle  problem (game). But
only continuity of the value functions was able to prove.

Dai and Yi \cite{Dai09} studied a one dimensional finite horizon
optimal investment problem with transaction costs, where the
terminal CRRA utility function is maximized. This problem is indeed
a singular stochastic control problem, and by constructing its
connection with a parabolic two obstacle problem, the authors proved
the $C^{1,2}$ ($C^{2,1}$ in their paper due to different order of
$x,t$) regularity of the value function. The risky asset in their
model follows a standard geometric Brownian motion, and the
objective function only involves a function of the terminal wealth.
Furthermore, that work lacks the verification theorem to show
optimality.

In this paper we consider a more general model which certainly
includes the one in \cite{Dai09}. We provide conditions under which
a classical solution to the associated HJB equation exists, thus the
usual viscosity solution approach is avoided. The organizations of
this paper is as follows. In the next section, we introduce the time
inhomogeneous singular stochastic control problem that we are
concerned with. Then its associated zero-sum game problem is
investigated in Section \ref{tiDynkin}, where conditions and
regularity of the value function are proved. This result is passed
on to the value function of the singular control in Section
\ref{tiSingular} where a verification theorem and the optimal
control policy are shown. In Section \ref{Opt_Jump} we study the
finite horizon singular control in the situation that the terminal
payoff function is more general, and show that the optimal policy
involves a possible jump at maturity. The discussion on the
regularity of the free boundaries of the continuation region in the
optimal control is provided in Section \ref{reg_B}, followed by
concluding remarks.

\section{Problem Formulation}\label{pform}

Given a probability space $(\Omega, \mathcal{F}_t, X,P_x)$, where
$P_x$ is a family of measures under which $X(t)=X_t$ is a one
dimensional Ito diffusion
\begin{displaymath}
dX_t=\mu(X_t)dt+\sigma(X_t)dB_t, \ \ X_s=x,
\end{displaymath} in which $B_t$ is a standard Brownian motion, we consider the model $\mu(x)=cx+d$ for some  constants $c,d$, and here $\sigma$ is a differentiable bounded function such that $\sigma(x)\geq\epsilon>0,\ \forall x$, for some constant $\epsilon$. Let $(A_t^{(1)},A_t^{(2)})=\mathcal{S}$ be two right continuous and nondecreasing $\mathcal{F}_t$ adapted
processes. We call $(A_t^{(1)},A_t^{(2)})$ the admissible controls
on $X_t$ and the controlled process then follows
\begin{equation}\label{omodel}
dX_t=\mu(X_t)dt+\sigma(X_t)dB_t+dA_t^{(1)}-dA_t^{(2)}, \ \ X_s=x.
\end{equation} Here we assume that $A_t^{(1)}-A_t^{(2)}$ is the minimal
decomposition of a bounded variation process into a difference of
two nondecreasing processes, and $\mathbb{S}$ is the set of all
admissible controls.

Consider the process $Z_t=(t,X_t)$ with time inhomogeneous
 cost function
\begin{eqnarray}\label{cost_s}
k_{\mathcal{S}}(z)=k_{\mathcal{S}}(s,x)&=&E_{(s,x)}\left\{\int_s^T H(t,X_t)dt\right.\nonumber\\
&&+\left.\int_s^T e^{-ct}f_1(t,X_t)dA_t^{(1)}+\int_s^T
e^{-ct}f_2(t,X_t)dA_t^{(2)}+G(X_T)\right\},
\end{eqnarray} where $H(\cdot)$ is called the holding cost,
$f_1(\cdot),f_2(\cdot)$ are the control costs and $G(\cdot)$ is the
terminal cost. We discount $f_1,f_2$ to time $0$ for convenience. If
one likes, the whole parts $e^{-ct}f_1(t,x),e^{-ct}f_2(t,x)$ can be
taken as the time inhomogeneous control costs.

One looks for a control policy $\mathcal{S}$ that minimizes
$k_{\mathcal{S}}(z)$, i.e.,
\begin{equation}\label{mincostw}
W(z)=\min_{\mathcal{S}\in\mathbb{S}} k_{\mathcal{S}}(z).
\end{equation}
This problem is called the time inhomogeneous singular control
problem. We solve this problem through a related time inhomogeneous
zero-sum game (Dynkin game).

\section{Time Inhomogeneous Dynkin Game}\label{tiDynkin}

Now consider the objective function of a time inhomogeneous zero-sum
game
\begin{eqnarray}\label{jcost}
J_z(\tau,\sigma)=J_{(s,x)}(\tau,\sigma)&=&E_{(s,x)}\left\{\int_s^{\tau\wedge\sigma\wedge T}e^{ct}h(t,Y_{t})dt\right.\nonumber\\
&&+1_{(\sigma <\tau\wedge T)}(-f_1(\sigma,Y_{\sigma}))+1_{(\tau <\sigma\wedge T)}f_2(\tau,Y_{\tau})\nonumber\\
&&\left.+1_{(T \leqslant\tau\wedge \sigma)}g(Y_{T})\right\},\ \
\tau\wedge\sigma\geqslant s,
\end{eqnarray} where  $Y_t$ follows
\begin{equation}\label{yproc}
dY_t=(\sigma(Y_t)\sigma'(Y_t)+\mu(Y_t))dt+\sigma(Y_t)dB_t, \ \
Y_s=x.
\end{equation} Two players, $P_1,P_2$, observe the process $Z_t=(t,Y_t)$ and
each of them can stop the process at any time $\tau,\sigma$
respectively before $T$. If $P_1$ stops the process at $\tau$, he
pays $P_2$ the amount
\begin{displaymath}\int_s^{\tau}e^{ct}h(t,Y_{t})dt+f_2(\tau,Y_{\tau}).\end{displaymath}For convenience, we call the part $\int_s^{\tau}e^{ct}h(t,Y_{t})dt$ the
dividend.  If $P_2$ stops the process at $\sigma$, he receives from
$P_1$ the amount
\begin{displaymath}\int_s^{\sigma}e^{ct}h(t,Y_{t})dt+(-f_1(\sigma,Y_{\sigma})).\end{displaymath}
If no one stops the game before $T$, $P_1$ pays $P_2$ the amount
\begin{displaymath}\int_s^{T}e^{ct}h(t,Y_{t})dt+g(Y_{T}).\end{displaymath}
Therefore $P_1$ tries to minimize his payment and $P_2$ tries to
maximize his income. The value of this game is thus given by
\begin{equation}
V(z)=V(s,x)=\inf_\tau\sup_\sigma J_z(\tau,\sigma)=\sup_\sigma
\inf_\tau J_z(\tau,\sigma),\ \ Y_s=x.
\end{equation}

\begin{assumption}\label{basic1}
The functions $f_1,f_2,g,h$ are all bounded and  continuous,
$M>f_2(t,y)> 0> -f_1(t,y)>-M$, $f_2(T,y)\geqslant g(y)\geqslant
-f_1(T,y), \forall y$, for some constant $M$.
\end{assumption}
The following is a version of Theorem 1 in \cite{Stettner11}.

\begin{theorem}\label{Vcontinuous} Under Assumption \ref{basic1}, the function $V(z)$
is a continuous function and is the value of the zero sum game. The
saddle point of this game has the form
\begin{eqnarray}
\hat{\tau}&=&\inf\{t\geqslant s:V(t,Y_t)=f_2(t,Y_t)\}\wedge T\nonumber\\
\hat{\sigma}&=&\inf\{t\geqslant s:V(t,Y_t)=-f_1(t,Y_t)\}\wedge T,\ \
Y_s=x.
\end{eqnarray}
\end{theorem}

Very often in this paper, $(s,x)$ denotes a starting point of the
process, and $(t,y)$ denotes an arbitrary point in the interested
region. Define $E=\{(t,y): 0\leqslant t\leqslant T,
-f_1(t,y)<V(t,y)<f_2(t,y)\}$, then this is the continuation region
of this zero sum game. Also define the sets $E_1=\{(t,y):0\leqslant
t\leqslant T, -f_1(t,y)=V(t,y)\}$ and $E_2=\{(t,y):0\leqslant
t\leqslant T, f_2(t,y)=V(t,y)\}$, and we notice that
$|V(t,y)|\leqslant M, \forall (t,y)\in [0,T]\times{\mathbb R}$.

Consider the infinitesimal generator of $Y_t$
\begin{equation}\label{infgen}
\mathcal{L}:=\frac{1}{2}\sigma^2\frac{\partial^2}{\partial
x^2}+(\sigma\sigma'+\mu)\frac{\partial}{\partial
x}+\frac{\partial}{\partial t},\ \ \ \ \left(':=\frac{d}{dx}\right),
\end{equation} then it is uniformly parabolic in our settings. We also
denote $L:=\frac{1}{2}\sigma^2\frac{\partial^2}{\partial
x^2}+(\sigma\sigma'+\mu)\frac{\partial}{\partial x}$. If the
variable $y$ is used in places of $x$, we use
$\frac{\partial}{\partial y}$ in (\ref{infgen}), and similarly if
$s$ is used in places of $t$, $\frac{\partial }{\partial s}$ is then
used.

\begin{assumption}\label{assumab}  $h(t,y)\in C([0,T)\times{\mathbb R})$, and $f_1,f_2$ are
$C^{1,2}([0,T]\times{\mathbb R})$ functions. $h(t,y)$  is strictly
increasing in $y$,  $f_1(t,y)$ is nondecreasing in $y$, $f_2(t,y)$
is nonincreasing in $y$, $\forall t\in[0,T]$. Furthermore, there
exist continuous  curves $a(t)$ and $b(t)$ with $a(t)<0<b(t)$,
$\forall t\in[0,T]$, such that
 \begin{eqnarray}
\mathcal{L}(-f_1(t,a(t)))+e^{ct}h(t,a(t))&=&0,\nonumber\\
\mathcal{L}f_2(t,b(t))+e^{ct}h(t,b(t))&=&0,\ \ \forall
t\in(0,T),\nonumber
\end{eqnarray} and $\mathcal{L}(-f_1(t,y))+e^{ct}h(t,y)$ and $\mathcal{L}f_2(t,y)+e^{ct}h(t,y)$ are both strictly
increasing in $y$.
\end{assumption} Without loss of generality, we set $h(t,0)=0, \forall t\in[0,T]$,
$g(0)=0$.

\begin{figure}[h!]
\centerline{\includegraphics[height=2.5in]{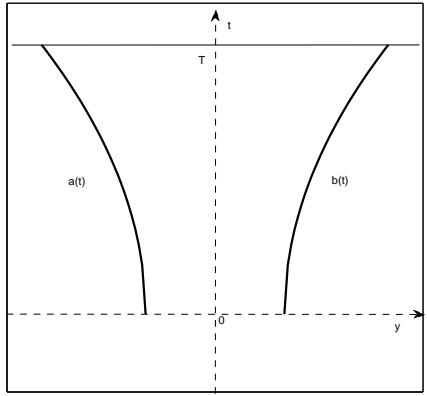}}\caption{Curves
$a(t),b(t)$} \label{region1}
\end{figure}

It is obvious that $\forall t\in(0,T)$,
\begin{eqnarray}
\mathcal{L}(-f_1(t,y))+e^{ct}h(t,y)&<&0,\ \ y<a(t),\nonumber\\
\mathcal{L}f_2(t,y)+e^{ct}h(t,y)&>&0,\ \ y>b(t),\nonumber\\
\mathcal{L}(-f_1(t,y))+e^{ct}h(t,y)&>&0,\ \ y>a(t),\nonumber\\
\mathcal{L}f_2(t,y)+e^{ct}h(t,y)&<&0,\ \ y<b(t).\nonumber
\end{eqnarray} Define the sets $Dab:=\{(t,y):0\leqslant t<T,a(t)<y<b(t)\}$,
$Da:=\{(t,y):0\leqslant t\leqslant T, y< a(t)\}$ and
$Db:=\{(t,y):0\leqslant t\leqslant T, y> b(t)\}$, then we have
\begin{proposition}\label{inDab}
For any point $(t,y)\in Dab$, it is not optimal for either player to
stop the game immediately.
\end{proposition}
\begin{proof}
Consider the optimal stopping strategy for player $P_2$, who wants
to maximize the income. If he stops the game immediately, the payoff
would be $-f_1(t,y)$ plus the expected dividend up to $t$. Construct
a small ball ${\bf B}_r(t,y)$ centered at $(t,y)$ such that ${\bf
B}_r(t,y)\subset Dab$, and consider the strategy to stop the game at
the first exit time $\tau_{{\bf B}_r}$ of the ball ${\bf B}_r(t,y)$.
By Dynkin's formula, the newly future expected payoff would be
\begin{eqnarray}\label{eqnbase}
&&E_{(t,y)}\left(\int_t^{\tau_{{\bf
B}_r}}e^{cu}h(u,Y_u)du-f_1(\tau_{{\bf
B}_r},Y_{\tau_{{\bf B}_r}})\right)\nonumber\\
&&=E_{(t,y)}\left(\int_t^{\tau_{{\bf
B}_r}}e^{cu}h(u,Y_u)+\mathcal{L}(-f_1)(u,Y_u)du\right)-f_1(t,y)>-f_1(t,y),
\end{eqnarray} since $e^{cu}h(u,Y_u)+\mathcal{L}(-f_1)(u,Y_u)>0$ in this region, and this is a contradiction. Therefore, it is not
optimal for $P_2$ to stop this game in $Dab$. Similarly we can show
that it is not optimal for $P_1$ either to stop the game in $Dab$.
\end{proof}
By a similar argument, it is easy to see that it is not optimal for
$P_2$ to stop the game in the region $Db$, and it is not optimal for
$P_1$ to stop the game in the region $Da$. Thus $Dab\subset E$,
$E_1\subset Da, E_2\subset Db$.

In what follows, $H^m$ denotes the $m$-th order Sobolev space, and
$H_0^m$ denotes the $m$-th order Sobolev space with compact support.

\begin{assumption}\label{AB}
There exist two points $A,B$  such that $A\leqslant a(T), B\geqslant
b(T)$, $g(y)$ is in $H^1((A,B))$. Furthermore,
$-f_1(T,y)=g(y),\forall y\leqslant A$ and $f_2(T,y)=g(y),\forall
y\geqslant B$.
\end{assumption} It is obvious that $A<0<B$ in our settings. Let $\tau_{AB}$ be the first exit time if the process exits the interested region through the line segment
$\{T\}\times[A,B]$. Obviously $\tau_{AB}=T$, but we use the notation
$\tau_{AB}<T$ to denote the event that the process exits the
interested region through the line segment $\{T\}\times[A,B]$, and
$T<\tau_{AB}$ denotes the event that the process exits the region
but not through the line segment $\{T\}\times[A,B]$, see Figure
\ref{region2}.

\begin{figure}[h!]
\centerline{\includegraphics[height=2in,
width=2.5in]{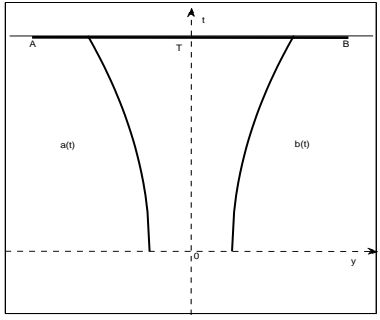}}\caption{Curves $a(t),b(t)$ and the
line segment $AB$} \label{region2}
\end{figure}

\begin{assumption}\label{tildeAB}
There exist constant curves $\tilde{A},\tilde{B}$ with $\tilde{A}<
A,\tilde{B}> B$, such that for any $y<\tilde{A}$,
\begin{displaymath}
E_{(t,y)}\left(\int_t^{\tau_a\wedge
T\wedge\tau_{AB}}(e^{cu}h(u,Y_u)+\mathcal{L}(-f_1(u,
Y_u)))du+1_{(\tau_a\wedge\tau_{AB}<T)}(2M)\right)\leqslant 0,
\end{displaymath} and for any $y>\tilde{B}$,
\begin{displaymath}
E_{(t,y)}\left(\int_t^{\tau_b\wedge
T\wedge\tau_{AB}}(e^{cu}h(u,Y_u)+\mathcal{L}f_2(u,Y_u))du-1_{(\tau_b\wedge\tau_{AB}<T)}(2M)\right)\geqslant
0,
\end{displaymath} where $\tau_a,\tau_b$ are the first hitting times
to the curves $a,b$, respectively.
\end{assumption}

Notice that $1_{(\tau_a\wedge\tau_{AB}<T)}$ denotes the event that
the process either hits the curve $a$ first or exits the interested
region through the line segment $\{T\}\times[A,B]$.

\begin{proposition}
It is optimal for $P_2$ to stop the game in the region
$[0,T]\times(-\infty,\tilde{A}]$ and it is optimal for $P_1$ to stop
the game in the region $[0,T]\times[\tilde{B},\infty)$.
\end{proposition}

\begin{proof} It suffices to prove the first half. Consider the
starting point $(s,x)\in [0,T]\times(-\infty,\tilde{A}]$. If $s=T$,
the result automatically holds. Now suppose $s<T$. Since
$\mathcal{L}(-f_1(t,y))+e^{ct}h(t,y)<0, y<a(t)$, and if $P_2$ stops
the game in the region $(s,T)\times(-\infty,a]$, the payoff would be
\begin{equation}\label{basearg}
E_{(s,x)}\left(\int_s^\tau(\mathcal{L}(-f_1(t,Y_t))+e^{ct}h(t,Y_t))dt\right)-f_1(s,x)<-f_1(s,x),
\end{equation} which is certainly not optimal. Now suppose $P_2$
would stop the game beyond the region $(s,T)\times(-\infty,a]$, the
payoff will be less than (since $M$ is an upper bound of $V$)
\begin{eqnarray}
&&E_{(s,x)}\left(\int_s^{\tau_a\wedge
T\wedge\tau_{AB}}e^{ct}h(t,Y_t)dt+1_{\tau_a<T\wedge\tau_{AB}}M+1_{T\leqslant\tau_a\wedge\tau_{AB}}(-f_1(T,Y_T))+1_{\tau_{AB}< T\wedge\tau_a}g(Y_{\tau_{AB}})\right)\nonumber\\
&&=E_{(s,x)}\left(\int_s^{\tau_a\wedge
T\wedge\tau_{AB}}e^{ct}h(t,Y_t)dt+1_{\tau_a<T\wedge\tau_{AB}}M-f_1(\tau_a\wedge T,Y_{\tau_a\wedge T})+1_{\tau_{AB}< T\wedge\tau_a}g(Y_{\tau_{AB}})\right.\nonumber\\
&&\ \ \ \ \ \left.+1_{\tau_a\wedge\tau_{AB}<T}f_1(\tau_a,Y_{\tau_a})\right)\nonumber\\
&&=E_{(s,x)}\left(\int_s^{\tau_a\wedge
T\wedge\tau_{AB}}(e^{ct}h(t,Y_t)+\mathcal{L}(-f_1(t,Y_t)))dt+1_{\tau_{AB}< T\wedge\tau_a}(g(Y_{\tau_{AB}})+f_1(\tau_{AB},Y_{\tau_{AB}}))\right.\nonumber\\
&&\ \ \ \ \  \left.+1_{\tau_a<T\wedge\tau_{AB}}(M+f_1(\tau_a,Y_{\tau_a}))\right)-f_1(s,x)\nonumber\\
&&\leqslant E_{(s,x)}\left(\int_s^{\tau_a\wedge
T\wedge\tau_{AB}}(e^{ct}h(t,Y_t)+\mathcal{L}(-f_1(t,Y_t)))dt+1_{\tau_a\wedge\tau_{AB}<T}(2M)\right)-f_1(s,x)\nonumber\\
&&\leqslant -f_1(s,x)\nonumber
\end{eqnarray} by Assumption \ref{tildeAB}. Therefore the optimal
strategy for $P_2$ at this starting point is to stop the game
immediately.
\end{proof}

Now it is clear that $[0,T]\times(-\infty,\tilde{A}]\subset E_1$,
$[0,T]\times[\tilde{B},\infty)\subset E_2$. In fact we can say
something more about the termination region of this game. Let
$\tilde{E}_1$ be the largest connected region in $E_1$ which
contains the set $[0,T]\times(-\infty,\tilde{A}]$, and $\tilde{E}_2$
 the largest connected region in $E_2$ which contains
$[0,T]\times[\tilde{B},\infty)$, then we have

\begin{proposition}
$\tilde{E}_1,\tilde{E}_2$ are both simply connected.
\end{proposition}
This proposition says there are no holes in $\tilde{E}_1$ or
$\tilde{E}_2$.
\begin{proof}
 By the
continuity of $V$ (Theorem \ref{Vcontinuous}) and the properties of
$f_1,f_2$, we know that the upper boundary (in $y$) of $\tilde{E}_1$
and the lower boundary (in $y$) of $\tilde{E}_2$ are continuous
curves. Suppose there is a point $(s,x)$ below the the upper
boundary of $\tilde{E}_1$ with $V(s,x)>-f_1(s,x)$, then again by the
continuity of $V$ and $f_1$, there is an open connected region ${\bf
B}_{(s,x)}$ containing $(s,x)$ such that $V(t,y)>-f_1(t,y),\ \forall
(t,y)\in {\bf B}_{(s,x)}$ and furthermore $\partial {\bf
B}_{(s,x)}\subset\tilde{E}_1$. Since in the region ${\bf B}_{(s,x)}$
the inequality $\mathcal{L}(-f_1(t,y))+e^{ct}h(t,y)<0$ holds while
on the boundary $\partial{\bf B}_{(s,x)}$ we have
$V(t,y)=-f_1(t,y)$, Dynkin's formula and an  argument similar to
(\ref{basearg})  easily imply that $V(s,x)<-f_1(s,x)$, hence a
contradiction. The rest of this proposition can be proved
analogously.
\end{proof}

Denote $\tilde{a}$  the upper boundary (in $y$) of $\tilde{E}_1$,
and $\tilde{b}$ the lower boundary of (in $y$) $\tilde{E}_2$. We
call a directed curve a $t_+$ curve if the points $(t,y)$ on this
curve is in a nondecreasing order in $t$ as the curve being lined up
from the initial point to the terminal point. A $t_-$ curve is
defined similarly. Notice that $\tilde{a}(T)=A$, $\tilde{b}(T)=B$.

\begin{proposition}
If there is a point $(t,y)$ with $y<a(t)$ such that
$V(t,y)>-f_1(t,y)$, then this point is connected to $Dab$ along a
$t_+$ curve in $E$. If there is a point $(t,y)$ with $y>b(t)$ such
that $V(t,y)<f_2(t,y)$, then this point is connected to $Dab$ along
a $t_+$ curve in $E$.
\end{proposition}
\begin{proof}
We just need to prove the first half as usual. If this point $(t,y)$
is not connected to $Dab$ along any $t_+$ curve in $E$, that means
starting from this point, the game should be stopped before the
process $Y_t$ hits the curve $a(t)$. Since in this region
$\mathcal{L}(-f_1(t,y))+e^{ct}h(t,y)<0$ holds and a similar argument
to (\ref{basearg})  easily implies that $V(t,y)\leqslant -f_1(t,y)$,
hence a contradiction.
\end{proof}

Let ${\mathbb U}=[0,T]\times[m,n]$ be any simply connected region
such that $m(s)\leqslant a(s),n(s)\geqslant b(s),s\in[0,T]$, are
functions of $s$ with possible jumps (in this case we can still draw
a continuous curve as the boundary of ${\bf U}$), and further
$m(T)=A,n(T)=B$. Denote $\tau_{\mathbb U}$ the first exit time of
${\mathbb U}$, and define
\begin{equation}\label{def_f}
F_{\mathbb U}(s,x)=E_{(s,x)}\left(\int_s^{\tau_{\mathbb
U}}e^{cu}h(u,Y_u)du+R(\tau_{\mathbb U},Y_{\tau_{\mathbb
U}})\right),\ (s,x)\in{\mathbb U},
\end{equation} where $R(\tau_{\mathbb U},Y_{\tau_{\mathbb U}})=g(Y_T)$ if
$\tau_{\mathbb U}=T$, $R(\tau_{\mathbb U},Y_{\tau_{\mathbb
U}})=-f_1(\tau_{\mathbb U},Y_{\tau_{\mathbb U}})$ if $\tau_{\mathbb
U}<T$ and $Y_{\tau_{\mathbb U}}\leqslant a(\tau_{\mathbb U})$, and
$R(\tau_{\mathbb U},Y_{\tau_{\mathbb U}})=f_2(\tau_{\mathbb
U},Y_{\tau_{\mathbb U}})$ if $\tau_{\mathbb U}<T$ and
$Y_{\tau_{\mathbb U}}\geqslant b(\tau_{\mathbb U})$.

For any point $(s,x)\in{\mathbb U}$ with $x<a(s)$, define an
associated cone $C_{(s,x)}$ as
\begin{displaymath}
C_{(s,x)}=\left\{(t,y)\in{\mathbb U}:s\leqslant t\leqslant T,
x<y\leqslant a(t),\frac{t-s}{y-x}\leqslant\eta\right\},
\end{displaymath} where $\eta>0$ is a constant. Let $\tilde{\mathbb U}={\mathbb U}\cup
C_{(s,x)}$, and define the function $F_{\tilde{\mathbb U}}(s,x)$
similarly as in (\ref{def_f}), we put a joint assumption on the
process $Y_t$ and the functions $f_1,f_2,h,g$ as follows:
\begin{assumption}\label{monoF1} For any $(s,x)\in{\mathbb U}$ with $x<a(s)$,
$F_{\tilde{\mathbb U}}(s,x)\geqslant F_{{\mathbb U}}(s,x)$.
\end{assumption}
Similarly for any point $(s,x)\in{\mathbb U}$ with $x>b(s)$, we may
define an associated cone $C_{(s,x)}$ as
\begin{displaymath}
C_{(s,x)}=\left\{(t,y)\in{\mathbb U}:s\leqslant t\leqslant T,
b(t)\leqslant y<x,\frac{t-s}{x-y}\leqslant\eta\right\}.
\end{displaymath}
\begin{assumption}\label{monoF2} For any $(s,x)\in{\mathbb U}$ with $x>b(s)$,
$F_{\tilde{\mathbb U}}(s,x)\leqslant F_{{\mathbb U}}(s,x)$.
\end{assumption}
\begin{remark}
In the one dimensional time homogeneous case, this assumption holds
true due to the a priori given conditions on $f_1,f_2,h$. But in the
case of time inhomogeneous or multiple dimensional case, the
situation is much more complicated and more conditions are needed.
\end{remark}

\begin{proposition}\label{Econ}
Under Assumptions \ref{monoF1}, \ref{monoF2}, the continuation
region $E$ is simply connected, and its lower boundary (in $y$) is
the curve $\tilde{a}$, its upper boundary (in $y$) is the curve
$\tilde{b}$.
\end{proposition}
\begin{proof}
We just prove the first half as usual. For any point $(s,x)\in
\bar{E}$ with $x<a(s)$, consider the new continuation region $E\cup
C_{(s,x)}$, then by Assumption \ref{monoF1}, we have $F_{\bar{E}\cup
C_{(s,x)}}\geqslant F_{\bar{E}(s,x)}=V(s,x)$. Since in the region
below the curve $a$, only player $P_2$ wants to stop the game who
wants to maximize his payoff, the new continuation region $E\cup
C_{(s,x)}$ is certainly better than $E$, if not identical, and this
is a contradiction since $E$ is assumed to be optimal. Thus
$\bar{E}\cup C_{(s,x)}=\bar{E}$. Since this holds for any point
$(s,x)\in\bar{E}$ with $x<a(s)$, we know that the region between the
curve $\tilde{a}$ and the curve $b$ (since $E$ contains the region
$Dab$) is simply connected.
\end{proof}

Now it can be seen that $\tilde{E}_1=E_1,\tilde{E}_2=E_2$.
Furthermore, Proposition \ref{Econ} implies that
$\tilde{a},\tilde{b}$ can be taken as functions of $t$, with
possible jumps (in this case we line up these points and still get
continuous curves). It is worth noticing that $\tilde{a}$ and the
curve $a$ are not necessarily identical, and $\tilde{b}$ and the
curve $b$ are not necessarily identical either. Denote
$D\tilde{a}\tilde{b}$ the open connected region between the two
curves $\tilde{a}$ and $\tilde{b}$, $D\tilde{a}$ the region to the
left (in $y$) of $\tilde{a}$ and $D\tilde{b}$ the region to the
right (in $y$) of $\tilde{b}$, then $D\tilde{a}\tilde{b}=E$
excluding the line segment $[A,B]$, $D\tilde{a}=E_1$,
$D\tilde{b}=E_2$, see Figure \ref{region3}.

\begin{figure}[h!]
\centerline{\includegraphics[height=2in,
width=2.5in]{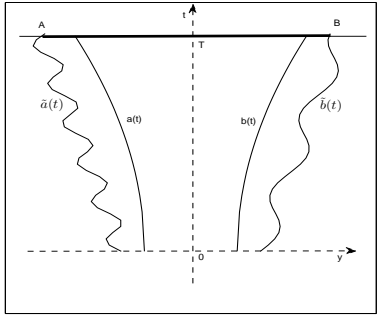}}\caption{Curves
$\tilde{a}(t),\tilde{b}(t)$} \label{region3}
\end{figure}

Certainly in the one dimensional time homogeneous case,
$\tilde{a},\tilde{b}$ are both constants, but in the time
inhomogeneous case, a $C^1$ regularity on $\tilde{a},\tilde{b}$ is
needed. Without loss of generality,
$[0,T]\times[\tilde{A},\tilde{B}]$ is assumed to be large enough to
contain the curves $\tilde{a}$ and $\tilde{b}$.

Now we can write $V(z)=J_z(\hat{\tau},\hat{\sigma})$. Let
$\phi(t,y)=-f_1(t,y)$ on the curve $y=\tilde{a}(t)$,
$\phi(t,y)=f_2(t,y)$ on the curve $y=\tilde{b}(t)$, $s\leqslant
t\leqslant T$, and $\phi(T,y)=g(y)$ on the line $t=T,
\tilde{a}(T)\leqslant y\leqslant \tilde{b}(T)$, we can further
rewrite $V$ as
\begin{equation}\label{VonBoundary}
V(s,x)=E_{(s,x)}\left[\int_s^{\tau_{D\tilde{a}\tilde{b}}}e^{ct}h(t,Y_t)dt+\phi(\tau_{D\tilde{a}\tilde{b}},Y_{\tau_{D\tilde{a}\tilde{b}}})\right],\
\ Y_s=x,
\end{equation} where $\tau_{D\tilde{a}\tilde{b}}$ is the first exit time of the
region $D\tilde{a}\tilde{b}$ defined as
$\tau_{D\tilde{a}\tilde{b}}(s,x)=\inf\{t>s:(t,Y_t)\notin
D\tilde{a}\tilde{b}\}$.

\begin{theorem}\label{Vreg} Assume Assumptions \ref{basic1} --
\ref{monoF2}. If $\tilde{a}(t),\tilde{b}(t)$ are $C^1$ functions of
$t$, then the function $V$ in (\ref{VonBoundary}) is bounded and
continuous and is the unique weak solution of the following problem:
\begin{eqnarray}
&&\mathcal{L}V(t,y)+e^{ct}h(t,y)=0,\ \ \ \ \ \ \ \ (t,y)\in D\tilde{a}\tilde{b},  \label{Vpde}\\
&&\mathcal{L}V(t,y)+e^{ct}h(t,y)<0,\ \ \ \ \ \ \ \ (t,y)\in D\tilde{a},\nonumber\\
&&\mathcal{L}V(t,y)+e^{ct}h(t,y)>0,\ \ \ \ \ \ \ \ (t,y)\in D\tilde{b},\nonumber\\
&&-f_1(t,y)<V(t,y)<f_2(t,y), \ \  (t,y)\in D\tilde{a}\tilde{b},\nonumber\\
&&V(t,y)=-f_1(t,y),\ \ \ \ (t,y)\in D\tilde{a},\nonumber\\
&&V(t,y)=f_2(t,y),\ \ \ \ \ (t,y)\in D\tilde{b},\nonumber\\
&&V(T,y)=g(y).\nonumber
\end{eqnarray}
Furthermore, by considering $V(t,y)$ as a mapping $V:[0,T]\to
H^2({\bf U})$, where ${\bf U}$ is the open interval
$(\tilde{A},\tilde{B})$, then $V\in L^2(0,T:H^2({\bf U}))\cap
L^\infty(0,T;H^1({\bf U}))$, $\frac{d V}{d t}\in L^2(0,T;L^2({\bf
U}))$.
\end{theorem}
\begin{proof}
Let $\tilde{V}(t,y)=V(t,y)+B(t,y)$ on $D\tilde{a}\tilde{b}$, where
\begin{displaymath}B(t,y)=f_1(t,\tilde{a}(t))+\frac{y-\tilde{a}(t)}{\tilde{b}(t)-\tilde{a}(t)}(-f_1(t,\tilde{a}(t))-f_2(t,\tilde{b}(t))),\
t\in[0,T], \tilde{a}(t)\leqslant y\leqslant
\tilde{b}(t),\end{displaymath} and consider the following partial
differential equation on the extended rectangular region
$[0,T]\times \bar{\bf U}=[0,T]\times[\tilde{A},\tilde{B}]$:
\begin{eqnarray}\label{vtilde}
\mathcal{L}\tilde{V}(t,y)&=&1_{((t,y)\in
D\tilde{a}\tilde{b})}\left(\mathcal{L}B(t,y)-e^{ct}h(t,y)\right)\triangleq \tilde{f}(t,y),\nonumber\\
\tilde{V}(t,y)&=&0, \ \ \ (t,y)\notin D\tilde{a}\tilde{b},\nonumber\\
\tilde{V}(T,y)&=&1_{(y\in[A,B])}\left(g(y)+B(T,y)\right)\triangleq
\tilde{g}(y).
\end{eqnarray}
By the conditions on the functions $f_1,f_2,g,h$ and the $C^1$
property of curves $\tilde{a},\tilde{b}$, it can be seen that (for
each $t$)
\begin{displaymath} \tilde{f}\in L^2(0,T;L^2({\bf U})),
\end{displaymath}
and $\tilde{g}\in H^1_0({\bf U})$. If we change the variable $t$ to
$T-t$, the terminal condition of (\ref{vtilde}) becomes initial
condition, and Theorem 5 in Chapter 7 of \cite{Evans10} can be
applied. Thus there exists a unique weak solution $\tilde{V}$ to the
problem (\ref{vtilde}). Furthermore, $\tilde{V}\in L^2(0,T;H^2({\bf
U}))\cap L^\infty(0,T;H^1_0({\bf U}))$, $\frac{d}{d t}\tilde{V}\in
L^2(0,T;L^2({\bf U})))$.
 Now it is easy to recover $V(t,y)$ by writing
 \begin{equation}V(t,y)=\left\{ \begin{aligned}&&\tilde{V}(t,y)-B(t,y),\ \ \  (t,y)\in
 D\tilde{a}\tilde{b},\\
 &&-f_1(t,y), \ \ \ \ \ (t,y)\in D\tilde{a},\\
 &&f_2(t,y),\ \ \ \ \  (t,y)\in
 D\tilde{b},\end{aligned}\label{Vrecov}\right.
 \end{equation} thus the regularity result on $V(t,y)$ follows.

Now expression (\ref{VonBoundary}) shows that $V$ is $h$-harmonic,
and if $V$ is a $C^{1,2}(D\tilde{a}\tilde{b})$ function (that is,
$C^{1,0}\cap C^{0,2}$), (\ref{Vpde}) should hold on $V$. By the
uniqueness of the weak solution, we know that this solution is given
by (\ref{VonBoundary}). The rest of this theorem has been proved in
Theorem \ref{Vcontinuous}.
\end{proof}

Since $V$ is a weak solution of the PDE (\ref{Vpde}), we can not
assure that $V$ is $C^{1,2}(D\tilde{a}\tilde{b})$, however, we can
conclude the following regularity result.

\begin{corollary}\label{bcond1} Assuming the conditions in Theorem \ref{Vreg}.
$V(t,y)$ in (\ref{Vpde}) is in $C^{0,1}([0,T)\times {\bf U})\cap
C([0,T]\times{\bf U})$.
\end{corollary}
\begin{proof}
This fact  is due to the regularity results $V\in L^2(0,T:H^2({\bf
U}))$ in Theorem \ref{Vreg}, and of course the continuous version is
chosen in this case. That is, if a function is in $H^2({\bf U})$,
then the $C^1$ version of this function is chosen.
\end{proof}
With this result, it is straight forward to conclude the following.
\begin{corollary}\label{bcond2}
Assuming the conditions in Theorem \ref{Vreg}, then for any
$t\in[0,T)$,
\begin{displaymath}
  \frac{\partial V}{\partial y}(t,\tilde{a}(t))=-\frac{\partial
f_1}{\partial y}(t,\tilde{a}(t)),\ \ \
  \frac{\partial V}{\partial y}(t,\tilde{b}(t))=\frac{\partial
f_2}{\partial y}(t,\tilde{b}(t)).
\end{displaymath}
\end{corollary}

Since $V(t,y)\geqslant -f_1(t,y)$ and $V(t,y)\leqslant f_2(t,y)$,
$\forall (t,y)\in(0,T)\times{\bf U}$, it is easy to see that
\begin{lemma}\label{bcond4} For any
$t\in[0,T)$,
\begin{eqnarray*}
&&\lim\inf_{\epsilon\to
0^+}\frac{V(t+\epsilon,\tilde{a}(t))-V(t,\tilde{a}(t))}{\epsilon}\geqslant
-f_{1t}(t,\tilde{a}(t)),\\
&&\lim\sup_{\epsilon\to
0^+}\frac{V(t,\tilde{a}(t))-V(t-\epsilon,\tilde{a}(t))}{\epsilon}\leqslant
-f_{1t}(t,\tilde{a}(t)),\\
\end{eqnarray*} and
\begin{eqnarray*}
&&\lim\sup_{\epsilon\to
0^+}\frac{V(t+\epsilon,\tilde{b}(t))-V(t,\tilde{b}(t))}{\epsilon}\leqslant
f_{2t}(t,\tilde{a}(t)),\\
&&\lim\inf_{\epsilon\to
0^+}\frac{V(t,\tilde{b}(t))-V(t-\epsilon,\tilde{b}(t))}{\epsilon}\geqslant
f_{2t}(t,\tilde{a}(t)).
\end{eqnarray*}
\end{lemma}

Furthermore, we can prove the following:
\begin{lemma}\label{bcond5} There is a constant $K>0$ such that for any $t\in[0,T)$,
\begin{eqnarray*}
&&\lim\sup_{\epsilon\to
0^+}\frac{V(t+\epsilon,\tilde{a}(t))-V(t,\tilde{a}(t))}{\epsilon}\leqslant
K,\\
&&-K\leqslant\lim\inf_{\epsilon\to
0^+}\frac{V(t,\tilde{a}(t))-V(t-\epsilon,\tilde{a}(t))}{\epsilon},\\
\end{eqnarray*} and
\begin{eqnarray*}
&&-K\leqslant\lim\inf_{\epsilon\to
0^+}\frac{V(t+\epsilon,\tilde{b}(t))-V(t,\tilde{b}(t))}{\epsilon},\\
&&\lim\sup_{\epsilon\to
0^+}\frac{V(t,\tilde{b}(t))-V(t-\epsilon,\tilde{b}(t))}{\epsilon}\leqslant
K.
\end{eqnarray*}
\end{lemma}
\begin{proof}
It suffices to prove the first inequality and the rest can be done
in a similar manner.  By the $C^1$ property of $\tilde{a}$, we can
find a small interval $(t,t+\epsilon)$ such that for all the points
$(s,\tilde{a}(t))$ with $s\in (t,t+\epsilon)$, either
$(s,\tilde{a}(t))\in D\tilde{a}$, or $(s,\tilde{a}(t))\in
D\tilde{a}\tilde{b}$. In the former case,
$V(s,\tilde{a}(t))=-f_1(s,\tilde{a}(t))$ and the result
automatically holds.

Now let us suppose $(s,\tilde{a}(t))\in D\tilde{a}\tilde{b}$,
$\forall s\in (t,t+\epsilon)$. We can write
\begin{displaymath}
\tilde{a}(t+\epsilon)=\tilde{a}(t)+\tilde{a}'(t)\epsilon+o(\epsilon).
\end{displaymath} Since $V\in C^{0,1}([0,T)\times {\bf U})\cap C([0,T]\times{\bf U})$
by Corollary \ref{bcond1}, $V(t,y)$ is uniformly Lipschitz in $y$,
thus there is a constant $M_1$ such that
\begin{displaymath}
|V(t+\epsilon,\tilde{a}(t))-V(t+\epsilon,\tilde{a}(t+\epsilon))|\leqslant
M_1|\tilde{a}'(t)\epsilon|+|o(\epsilon)|.
\end{displaymath}
But
$V(t+\epsilon,\tilde{a}(t+\epsilon))=-f_1(t+\epsilon,\tilde{a}(t+\epsilon))$,
hence
\begin{displaymath}
V(t+\epsilon,\tilde{a}(t))\leqslant
-f_1(t+\epsilon,\tilde{a}(t+\epsilon))+M_1|\tilde{a}'(t)\epsilon|+|o(\epsilon)|.
\end{displaymath} Therefore
\begin{eqnarray*}
&&\lim\sup_{\epsilon\to
0^+}\frac{V(t+\epsilon,\tilde{a}(t))-V(t,\tilde{a}(t))}{\epsilon}\\
&&\leqslant\lim\sup_{\epsilon\to
0^+}\frac{-f_1(t+\epsilon,\tilde{a}(t+\epsilon))+M_1|\tilde{a}'(t)\epsilon|+|o(\epsilon)|-(-f_1(t,\tilde{a}(t)))}{\epsilon}\\
&&=\lim_{\epsilon\to
0^+}\frac{-f_1(t+\epsilon,\tilde{a}(t+\epsilon))-(-f_1(t,\tilde{a}(t)))}{\epsilon}+M_1|\tilde{a}'(t)|<\infty
\end{eqnarray*} by the smoothness of $f_1$.
\end{proof}

\section{Time Inhomogeneous Singular Control}\label{tiSingular}
In this section we assume the conditions in Theorem \ref{Vreg}. Let
the functions $V,h$ be given in Theorem \ref{Vreg}, and define the
function $W(s,x)$ on $[0,T]\times {\bf U}$ as
\begin{equation}\label{wdef}
W(s,x)=\int_0^x e^{-cs}V(s,y)dy,
\end{equation}we shall investigate the properties of
$W$. Firstly it is easy to see that $W\in C^{0,2}([0,T)\times {\bf
U})\cap C^{0,1}([0,T]\times {\bf U})$ by Corollary \ref{bcond1}, and
on $[0,T)\times {\bf U}$,
\begin{displaymath}
\frac{\partial W}{\partial x}=e^{-cs}V,\ \frac{\partial^2
W}{\partial x^2}=\frac{\partial e^{-cs}V}{\partial x}.
\end{displaymath}

What is not very obvious is the following proposition:
\begin{proposition}
$W\in C^{1,0}([0,T)\times{\bf U})$.
\end{proposition}
\begin{proof}
For any $(s,x)\in D\tilde{a}\tilde{b}$, by (\ref{wdef}) and
(\ref{Vpde}) and using integration by parts, we get
\begin{eqnarray*}
&&\frac{\partial W}{\partial s}(s,x)=-c\int_0^x
e^{-cs}V(s,y)dy+\int_0^x e^{-cs}V_s(s,y)dy\\
&&=-c\int_0^x e^{-cs}V(s,y)dy\\
&&-\int_0^x
e^{-cs}\left(\frac{1}{2}\sigma^2(y)V_{xx}(s,y)+(\sigma(y)\sigma'(y)+\mu(y))V_x(s,y)+e^{cs}h(s,y)\right)dy\\
&&=-c
W(s,x)-e^{-cs}\left(\frac{1}{2}\sigma^2(x)V_x(s,x)-\frac{1}{2}\sigma^2(0)V_x(s,0)\right)\\
&&-\int_0^x \left(e^{-cs}\mu(y)V_x(s,y)+h(s,y)\right)dy,
\end{eqnarray*} which is bounded and continuous.

Now we can send $x\to\tilde{a}(s)$, and by the continuity of $W$ and
$V_x$, we know that $\frac{\partial W}{\partial
s}(s,\tilde{a}(s)^+)$ exists which is
\begin{eqnarray*}
&&\frac{\partial W}{\partial s}(s,\tilde{a}(s)^+)=-c
W(s,\tilde{a}(s))-e^{-cs}\left(\frac{1}{2}\sigma^2(\tilde{a}(s))V_x(s,\tilde{a}(s))-\frac{1}{2}\sigma^2(0)V_x(s,0)\right)\\
&&-\int_0^{\tilde{a}} \left(e^{-cs}\mu(y)V_x(s,y)+h(s,y)\right)dy.
\end{eqnarray*} Similarly $\frac{\partial W}{\partial s}(s,\tilde{b}(s)^-)$
exists.

If $(s,x)$ is in the interior of $D\tilde{a}$, $V(s,x)=-f_1(s,x)$,
and if $(s,x)$ is in the interior of $D\tilde{b}$,
$V(s,x)=f_2(s,x)$. Now let us consider the quantity $\frac{\partial
W}{\partial s}(s,\tilde{a}(s))$ (and similarly $\frac{\partial
W}{\partial s}(s,\tilde{b}(s))$). We have shown that $\frac{\partial
W}{\partial s}(s,\tilde{a}(s)^+)$ is well defined, and the
difference between $\frac{\partial W}{\partial s}(s,\tilde{a}(s)^+)$
and $\frac{\partial W}{\partial s}(s,\tilde{a}(s))$ happens in the
following integral over a set of zero Lebesgue measure
\begin{eqnarray*}
&&\lim_{\epsilon\to
0}\int_{\tilde{a}(s)^+}^{\tilde{a}(s)}\frac{e^{-c(s+\epsilon)}V(s+\epsilon,y)-e^{-cs}V(s,y)}{\epsilon}dy\\
&&=-c\int_{\tilde{a}(s)^+}^{\tilde{a}(s)}e^{-cs}V(s,y)dy+e^{-cs}\lim_{\epsilon\to
0}\int_{\tilde{a}(s)^+}^{\tilde{a}(s)}\frac{V(s+\epsilon,y)-V(s,y)}{\epsilon}dy\\
&&=e^{-cs}\lim_{\epsilon\to
0}\int_{\tilde{a}(s)^+}^{\tilde{a}(s)}\frac{V(s+\epsilon,y)-V(s,y)}{\epsilon}dy.
\end{eqnarray*}
Here Lemma \ref{bcond4} and \ref{bcond5} come into play and we
conclude that the above quantity is zero. Thus $\frac{\partial
W}{\partial s}(s,\tilde{a}(s)^+)=\frac{\partial W}{\partial
s}(s,\tilde{a}(s))$ the continuity of $W_s$ follows.
\end{proof}

\begin{remark}
As a conclusion, by integrating $V(s,y)$ as in (\ref{wdef}), we not
only obtain the $C^{0,2}$ regularity of $W(s,x)$, but also the
$C^{1,0}$ regularity of $W(s,x)$.
\end{remark}

For each $s\in[0,T)$, define the function
\begin{displaymath}
C(s)=-\frac{1}{2}\sigma^2(\tilde{a}(s))\frac{\partial^2 W}{\partial
x^2}(s,\tilde{a}(s))-\mu(\tilde{a}(s))\frac{\partial W}{\partial
x}(s,\tilde{a}(s))-\frac{\partial W}{\partial
s}(s,\tilde{a}(s))-\int_0^{\tilde{a}(s)}h(s,y)dy,
\end{displaymath} then obviously $C(s)$ is a continuous function of
$s$. If we define \begin{displaymath}H(s,x)=\int_0^xh(s,y)dy +
C(s),\end{displaymath} then since $h(s,y)$ can be arbitrarily chosen
(under the assumptions in Theorem \ref{Vreg}), $H$ can also be an
arbitrary function being taken as the holding cost in
(\ref{cost_s}). Similarly, in order to link the terminal costs of
the singular control and the Dynkin game, we put
\begin{displaymath}G(x)=\int_0^x e^{-cT}g(y)dy.\end{displaymath}

Now that on $D\tilde{a}\tilde{b}$, $V$ satisfies
\begin{displaymath}
\frac{1}{2}\sigma^2(x)\frac{\partial^2 V}{\partial
x^2}+(\sigma(x)\sigma'(x)+\mu(x))\frac{\partial V}{\partial
x}+\frac{\partial V}{\partial s}+e^{cs}h(s,x)=0,
\end{displaymath} which is equivalent to
\begin{equation}\label{tVpde}
\frac{1}{2}\sigma^2(x)e^{-cs}\frac{\partial^2 V}{\partial
x^2}+(\sigma(x)\sigma'(x)+\mu(x))e^{-cs}\frac{\partial V}{\partial
x}+e^{-cs}\frac{\partial V}{\partial s}+h(s,x)=0.
\end{equation} We may now rewrite (\ref{tVpde}) as
\begin{displaymath}
\frac{1}{2}\sigma^2(x)\frac{\partial^2 e^{-cs}V}{\partial
x^2}+\sigma(x)\sigma'(x)\frac{\partial e^{-cs}V}{\partial
x}+\mu(x)\frac{\partial e^{-cs}V}{\partial x}+\frac{\partial
e^{-cs}V}{\partial s} + ce^{-cs}V+h(s,x)=0.
\end{displaymath} Since $\mu'(x)=c$, we further get
\begin{equation}\label{ttvpde}
\frac{1}{2}\sigma^2(x)\frac{\partial^2 e^{-cs}V}{\partial
x^2}+\sigma(x)\sigma'(x)\frac{\partial e^{-cs}V}{\partial
x}+\mu(x)\frac{\partial e^{-cs}V}{\partial x}+
\mu'(x)e^{-cs}V+\frac{\partial e^{-cs}V}{\partial s} +h(s,x)=0.
\end{equation} Integrating (\ref{ttvpde}) from $0$ to $x$ we get
\begin{equation}\label{wpde}
\frac{1}{2}\sigma^2(x)\frac{\partial^2 W}{\partial
x^2}+\mu(x)\frac{\partial W}{\partial x}+\frac{\partial W}{\partial
s}+H=0,
\end{equation} which is the HJB equation of the value function of the singular control problem.

Firstly (\ref{wpde}) certainly holds at $x=\tilde{a}(s)$ by
construction. That is, if we construct, for each $s$, a function
$U_s(x)=\frac{1}{2}\sigma^2(x)\frac{\partial^2 W}{\partial
x^2}+\mu(x)\frac{\partial W}{\partial x}+\frac{\partial W}{\partial
s}+H$, then $U_s(\tilde{a}(s))=0$. Furthermore, $U'_s(x)=0$ for
$x\in(\tilde{a}(s),\tilde{b}(s))$ by (\ref{ttvpde}), so (\ref{wpde})
holds for $x\in(\tilde{a}(s),\tilde{b}(s))$. Actually we can say
something more about $W$. By Theorem \ref{Vreg}, we see that
$U_s'(x)<0$ for $x<\tilde{a}(s)$, and $U_s'(x)>0$ for
$x>\tilde{b}(s)$, thus
\begin{equation}\label{wpde1}
\frac{1}{2}\sigma^2(x)\frac{\partial^2 W}{\partial
x^2}+\mu(x)\frac{\partial W}{\partial x}+\frac{\partial W}{\partial
s}+H>0,\ \ x\in
[\tilde{A},\tilde{a}(s))\cup(\tilde{b}(s),\tilde{B}].
\end{equation}
As a summary, we have the following theorem:
\begin{theorem}\label{wmain}
 Assume Assumptions \ref{basic1} --
\ref{monoF2}. If $\tilde{a}(s),\tilde{b}(s)$ are $C^1$ functions of
$s$, then there exists a $C^{1,2}([0,T)\times {\bf U})\cap
C^{0,1}([0,T]\times {\bf U})$ function $W(s,x)$ which satisfies
\begin{eqnarray}
\frac{1}{2}\sigma^2(x)\frac{\partial^2 W}{\partial
x^2}(s,x)+\mu(x)\frac{\partial W}{\partial x}(s,x)+\frac{\partial
W}{\partial s}(s,x)+H(s,x)=0,\ \
x\in(\tilde{a}(s),\tilde{b}(s)),&&\label{weqn}\\
\frac{1}{2}\sigma^2(x)\frac{\partial^2 W}{\partial
x^2}(s,x)+\mu(x)\frac{\partial W}{\partial x}(s,x)+\frac{\partial
W}{\partial s}(s,x)+H(s,x)>0,\ \ x\in
[\tilde{A},\tilde{a}(s))\cup(\tilde{b}(s),\tilde{B}].&&\label{wineq}\\
\frac{\partial W}{\partial x}(s,x)=-e^{-cs}f_1(s,x),\
x\leqslant\tilde{a}(s),\ \ \ \frac{\partial W}{\partial
x}(s,x)=e^{-cs}f_2(s,x),\ x\geqslant\tilde{b}(s), \label{bcond3}&&\\
-e^{-cs}f_1(s,x)<\frac{\partial W}{\partial
x}(s,x)<e^{-cs}f_2(s,x),\ \
x\in(\tilde{a}(s),\tilde{b}(s)),&&\\
W(T,x)=G(x),\ \ x\in {\bf U}.&&
\end{eqnarray}
\end{theorem} Notice that $W$ is $C^{1,2}$ on $[0,T)\times {\bf U}$, and
$f_1,f_2$ are smooth functions, more properties can be derived from
(\ref{bcond3}), for example, we have
\begin{displaymath}
\frac{\partial^2 W}{\partial x^2}(s,x)=-e^{-cs}\frac{\partial
f_1}{\partial x}(s,x), \ \ x\leqslant\tilde{a}(s),\ \
\frac{\partial^2 W}{\partial x^2}(s,x)=e^{-cs}\frac{\partial
f_2}{\partial x}(s,x), \ \ x\geqslant\tilde{b}(s).
\end{displaymath}

Let $\mathcal{S}^*$ be the control policy to reflect the process
$X_t$ within the region $D\tilde{a}\tilde{b}$. That is, whenever the
process touches the curve $\tilde{a}$, $A_t^{(1)}$ in (\ref{omodel})
increases and pushes $X_t$ back to $D\tilde{a}\tilde{b}$, with
smallest possible effort; whenever the process touches the curve
$\tilde{b}$, $A_t^{(2)}$ increases and pushes $X_t$ back to
$D\tilde{a}\tilde{b}$, with smallest possible effort.

We call $\mathcal{S}=(A_t^{(1)},A_t^{(2)})$ an admissible control if
\begin{enumerate}
\item There is a filtered measurable space $(\Omega,\{\mathcal{F}_t\}_{t\geqslant
0})$ subject to usual conditions and a probability measure $\{P_{
x}\}$ on it such that $\{{ X}_t\}_{t\geqslant 0}$ is an
$\{\mathcal{F}_t\}$-adapted process; $(A_t^{(1)},A_t^{(2)})$ are
right continuous $\{\mathcal{F}_t\}$ measurable processes with
bounded variation, and $A_t^{(1)}-A_t^{(2)}$ is the minimal
decomposition of a bounded variation process into a difference of
two nondecreasing processes.
\item There is $\{\mathcal{F}_t\}$ adapted Brownian motion $B_t$ such
that the following equation
\begin{displaymath}
dX_t=\mu(X_t)dt+\sigma(X_t)dB_t+dA_t^{(1)}-dA_t^{(2)}, \ \ X_s=x,
\end{displaymath} holds $P_x$ a.s., $\forall s\in[0,T)$,  and the controlled
process $X_t$ is a reflected $\mathcal{F}_t$ measurable process
within a compact region containing $(s,x)$ in $[0,T]\times{\mathbb
R}$ with continuous boundary.
\end{enumerate}
\begin{remark} The probability space $\Omega$ with the filtration
$\{\mathcal{F}_t\}$ is not fixed a priori. It is part of an
admissible policy. The filtration $\{\mathcal{F}_t\}$ is assumed to
be right continuous and $\mathcal{F}_0$ is assumed to contain every
$P_{ x}$-negligible set.
\end{remark}

In what follows we shall prove that the control policy
$\mathcal{S}^*$ is optimal. Firstly, it can be seen that the
reflected SDE in (\ref{omodel})  has a unique  solution for each
$(s,x)\in D\tilde{a}\tilde{b}$, see, e.g., Theorem 2.6 in
\cite{Burdzy09}. And obviously the controlled reflected diffusion
$X_t$ is $\{\mathcal{F}_t\}$ measurable.

If  the control $(A_t^{(1)},A_t^{(2)})$ involves  possible jumps at
time $t$, we use
\begin{displaymath}
\Delta A_t^{(i)}:=A_t^{(i)}-A_{t-}^{(i)},\ \ t\in[0,T),\ i=1,2,
\end{displaymath} to denote the jumps in the control, and use
$A_t^{(i),c} (i=1,2)$ to denote the continuous part of the processes
$A_t^{(i)},i=1,2$.. Since $A_t^{(1)},A_t^{(2)}$ are the minimal
decomposition of a bounded variation process into a difference of
two nondecreasing processes, we have $\Delta A_t^{(1)}\Delta
A_t^{(2)}=0,\forall t\in[0,T)$. In a similar manner we define
\begin{displaymath}
\Delta X_t:=X_t-X_{t-}, \ \ \Delta W(t,X_t):=W(t,X_t)-W(t-,X_{t-}),\
\ \forall t\in[0,T).
\end{displaymath}

We assume that the definition of $W(s,x)$ in (\ref{wdef}) still
holds for $(s,x)\notin D\tilde{a}\tilde{b}$ by noticing that
$V(s,x)=-f_1(s,x),x<\tilde{a}(s)$ and
$V(s,x)=f_2(s,x),x>\tilde{b}(s)$. Then we have the following
theorem.
\begin{theorem}\label{veri}
Assume Assumptions \ref{basic1} -- \ref{monoF2} and that
$\tilde{a}(s),\tilde{b}(s)$ are $C^1$ functions of $s\in[0,T]$. Let
$k_{\mathcal S}(z)=k_{\mathcal S}(s,x)$ be given by the following
\begin{eqnarray}\label{kcostjump} k_{\mathcal{S}}(z)=&&E_{(s,x)}\left(\int_s^T H(t, X_t)dt+G(X_T)\right)\\
&&+E_{(s,x)}\left(\int_s^T e^{-ct}\left(f_1(t,
X_t)dA_t^{(1),c}+f_2(t, X_t)dA_t^{(2),c}\right)\right)\nonumber\\
&&+E_{(s,x)}\left(\sum_{s\leqslant t\leqslant
T}e^{-ct}\left(\int_{X_{nt^-}}^{X_{nt^-}+\Delta A_t^{(1)}} f_1(t,
y)dy\right.\right.\nonumber\\
&&\left.\left.+\int_{X_{nt^-}-\Delta A_t^{(2)}}^{X_{nt^-}} f_2(t,
y)dy\right)\right),\nonumber
\end{eqnarray}
then
\begin{enumerate}

\item For any admissible policy $\mathcal{S}$, $W(s, x)\leqslant k_{\mathcal
S}(s, x),\ \forall (s,x)\in [0,T]\times\mathbb{R}$.

\item $W(s, x)= k_{{\mathcal S}^*}(s,
x),\ \forall (s, x)\in[0,T]\times\mathbb{R}$.
\end{enumerate}
\end{theorem}

\begin{proof}
Applying the generalized Ito's formula to $W(s,x)$ yields, for any
stopping time $\tau\in[s,T)$ and admissible control $\mathcal{S}$,
\begin{eqnarray}\label{wito}
W(\tau,X_\tau)&=& W(s,x)+\int_s^\tau\left(\frac{\partial W}{\partial t}(t,X_t)+\mu(X_t)\frac{\partial W}{\partial x}(t,X_t)+\frac{1}{2}\sigma^2(X_t)\frac{\partial^2 W}{\partial x^2}(t,X_t)\right)dt \nonumber\\
&&+ \int_s^\tau \frac{\partial W}{\partial
x}(t,X_t)\sigma(X_t)dB_t\nonumber\\
&&+\int_s^\tau\frac{\partial W}{\partial
x}(t,X_t)(dA_t^{(1),c}-dA_t^{(2),c})+\sum_{s\leqslant t< \tau}\Delta
W(t,X_t) .
\end{eqnarray}
by taking expectation of both sides of (\ref{wito}) we get
\begin{eqnarray}\label{wdynkin}
W(s,x)&=&E_{(s,x)}W(\tau,X_\tau)\nonumber\\
&&-E_{(s,x)}\left(\int_s^\tau\left(\frac{\partial W}{\partial
t}(t,X_t)+\mu(X_t)\frac{\partial W}{\partial
x}(t,X_t)+\frac{1}{2}\sigma^2(X_t)\frac{\partial^2 W}{\partial
x^2}(t,X_t)\right)dt\right)\nonumber\\
&&-E_{(s,x)}\int_s^\tau\frac{\partial W}{\partial
x}(t,X_t)(dA_t^{(1),c}-dA_t^{(2),c})-E_{(s,x)}\sum_{s\leqslant t<
\tau}\Delta W(t,X_t).\nonumber
\end{eqnarray}
We can rewrite $k_{\mathcal{S}}(s,x)$ as
\begin{eqnarray} k_{\mathcal{S}}(s,x)=&&E_{(s,x)}\left(\int_s^\tau H(t, X_t)dt+k_{\mathcal{S}}(\tau,X_\tau)\right)\\
&&+E_{(s,x)}\left(\int_s^\tau e^{-ct}\left(f_1(t,
X_t)dA_t^{(1),c}+f_2(t, X_t)dA_t^{(2),c}\right)\right)\nonumber\\
&&+E_{(s,x)}\left(\sum_{s\leqslant
t<\tau}e^{-ct}\left(\int_{X_{nt^-}}^{X_{nt^-}+\Delta A_t^{(1)}}
f_1(t,
y)dy\right.\right.\nonumber\\
&&\left.\left.+\int_{X_{nt^-}-\Delta A_t^{(2)}}^{X_{nt^-}} f_2(t,
y)dy\right)\right),\ \ \tau\in[s,T).\nonumber
\end{eqnarray}
Therefore,
\begin{eqnarray}\label{kwdiff}
&&k_{\mathcal{S}}(s,x)-W(s,x)\nonumber\\
&&=E_{(s,x)}\left(k_{\mathcal{S}}(\tau,X_\tau)-W(\tau,X_\tau)\right)\nonumber\\
&& +E_{(s,x)}\int_s^\tau\left( H(t,X_t)+\frac{\partial W}{\partial
t}(t,X_t)+\mu(X_t)\frac{\partial W}{\partial
x}(t,X_t)+\frac{1}{2}\sigma^2(X_t)\frac{\partial^2 W}{\partial
x^2}(t,X_t)\right)dt \nonumber\\
&&+ E_{(s,x)}\int_s^\tau \left(e^{-ct}f_1(t,X_t)+\frac{\partial W}{\partial x}(t,X_t)\right)dA_t^{(1),c}\nonumber\\
&&+E_{(s,x)}\int_s^\tau\left(e^{-ct}f_2(t,X_t)-\frac{\partial
W}{\partial
x}(t,X_t)\right)dA_t^{(2),c}\nonumber\\
&&+E_{(s,x)}\sum_{s\leqslant t< \tau}\Delta
W(t,X_t)\nonumber\\
&&+ E_{(s,x)}\left(\sum_{s\leqslant
t<\tau}\left(\int_{X_{nt^-}}^{X_{nt^-}+\Delta A_t^{(1)}}
e^{-ct}f_1(t, y)dy+\int_{X_{nt^-}-\Delta A_t^{(2)}}^{X_{nt^-}}
e^{-ct}f_2(t, y)dy\right)\right).
\end{eqnarray} By Theorem \ref{wmain}, the second, third and fourth
expectations in (\ref{kwdiff}) are all nonnegative, and this holds
true as $\tau\to T$.

Define the sets
\begin{displaymath}\Gamma_+=\{t\in[s,T]: \Delta
A_t^{(1)}>0\},\quad \Gamma_-=\{t\in[s,T]: \Delta A_t^{(2)}>0\},
\end{displaymath} then $\Gamma_+\cap\Gamma_-=\phi$. If we send $\tau\to T$, $k_{\mathcal{S}}(T-,X_{T-})$ can be written as $G(X_T)$ plus the possible jump of control at $T$.
Thus by sending $\tau\to T$ and the fact that
$k_{\mathcal{S}}(T,X_T)=G(X_T)=W(T,X_T)$ (after a possible jump at
$T$), we can rewrite the remaining parts in (\ref{kwdiff}) as
\begin{eqnarray}\label{rewrKW}
&&E_{(s,x)}\left(G(X_T)-W(T,X_T)\right)+E_{(s,x)}\sum_{s\leqslant
t\leqslant T}\Delta W(t,X_t) \nonumber\\
&&+ E_{(s,x)}\left(\sum_{s\leqslant t\leqslant
T}\left(\int_{X_{nt^-}}^{X_{nt^-}+\Delta A_t^{(1)}} e^{-ct}f_1(t,
y)dy+\int_{X_{nt^-}-\Delta A_t^{(2)}}^{X_{nt^-}}
e^{-ct}f_2(t, y)dy\right)\right)\nonumber\\
&&=E_{(s,x)}\left[\sum_{t\in\Gamma_+}\int_{X_{nt^-}}^{X_{nt^-}+\Delta
A_t^{(1)}}\left(\frac{\partial W}{\partial x}(t,y)+ e^{-ct}f_1(t,
y)\right)dy\right]\nonumber\\
&&+E_{(s,x)}\left[\sum_{t\in\Gamma_-}\int_{X_{nt^-}-\Delta
A_t^{(2)}}^{X_{nt^-}} \left(-\frac{\partial W}{\partial
x}(t,y)+e^{-ct}f_2(t, y)\right)dy\right].
\end{eqnarray} Once again by Theorem \ref{wmain}, these two
integrals are nonnegative, hence $k_{\mathcal{S}}(s,x)\geqslant
W(s,x)$.

If ${\mathcal{S}}=\mathcal{S}^*$ except a possible jump at time $s$
and the controlled process is the reflected diffusion in
$D\tilde{a}\tilde{b}$, then the second expectation in (\ref{kwdiff})
is zero. Since $A_t^{(1)}$ only increases when $X_t$ hits
$\tilde{a}$, and $A_t^{(2)}$ only increases when $X_t$ hits
$\tilde{b}$, by (\ref{bcond3}) the third and fourth expectations in
(\ref{kwdiff}) are both zeros. The remaining parts in (\ref{kwdiff})
is expressed in (\ref{rewrKW}) which indicates that under
$\mathcal{S}^*$ it is equal to zero. Therefore $W(s, x)=
k_{{\mathcal S}^*}(s, x),\ \forall (s, x)\in[0,T]\times\mathbb{R}$.
\end{proof}
The proof of Theorem \ref{veri} also implies that the optimal
admissible control is unique.
\begin{remark}
The term ``a possible jump at time $s$'' means if the initial state
of the process is outside of the region $D\tilde{a}\tilde{b}$, apply
a control $\Delta A_s^{(1)}$ or $\Delta A_s^{(2)}$ to immediately
bring it into $D\tilde{a}\tilde{b}$.
\end{remark}

\section{Optimal Control with A More General Terminal Cost}\label{Opt_Jump}
Consider again the cost function of the singular control problem
\begin{eqnarray}\label{cost_ag}
k_{\mathcal{S}}(z)=k_{\mathcal{S}}(s,x)&=&E_{(s,x)}\left\{\int_s^T H(t,X_t)dt\right.\nonumber\\
&&+\left.\int_s^T e^{-ct}f_1(t,X_t)dA_t^{(1)}+\int_s^T
e^{-ct}f_2(t,X_t)dA_t^{(2)}+G(X_T)\right\},
\end{eqnarray} where $G'(x)=e^{-cT}g(x)$. In Assumption \ref{basic1} we
have the condition
\begin{equation}\label{reg_G}-f_1(T,x)\leqslant g(x)\leqslant f_2(T,x),\ \forall x.\end{equation}
In this section, we shall consider a more general terminal cost
$G(x)$ such that (\ref{reg_G}) does not necessarily hold. We shall
see that the optimal control involves a jump at the terminal $T$. In
this section the functions $f_1,f_2,g,h$ are still assumed to be
bounded and continuous with $f_2>0>-f_1$, and we further assume
Assumptions \ref{assumab} and \ref{tildeAB}.

The following is a relaxed condition on $g(x)$ which provides a more
general terminal cost function $G(x)=\int_0^x e^{-cT}g(y)dy$.
\begin{assumption}\label{ABmod}
There exist two points $A,B$  such that $A<0< B$, $g(x)$ is in
$H^1((A,B))$. Furthermore, $\forall x< A$, $-f_1(T,x)>g(x)$,
$\forall x> B$, $f_2(T,x)<g(x)$, and $\forall x\in[A,B]$,
$-f_1(T,x)\leqslant g(x)\leqslant f_2(T,x)$.
\end{assumption} Comparing this assumption with Assumption \ref{AB},
it can be seen that we no longer require the conditions $A\leqslant
a(T)$ or $B\geqslant b(T)$.

Define $\tilde{G}(x)$ as
\begin{equation}\label{G-tilde}
\tilde{G}(x)=\min_{y_1\geqslant 0,y_2\geqslant
0}\left[G(x+y_1)+\int_x^{x+y_1}e^{-cT}f_1(T,u)du,\
G(x-y_2)+\int_{x-y_2}^x e^{-cT}f_2(T,u)du\right],
\end{equation} then we have

\begin{proposition}
For any $x$ with $g(x)<-f_1(T,x)$, $\tilde{G}(x)<G(x)$ and
$\tilde{g}(x):=e^{cT}\tilde{G}'(x)=-f_1(T,x)$, and for any $x$ with
$g(x)>f_2(T,x)$, $\tilde{G}(x)<G(x)$ and $\tilde{g}(x)=f_2(T,x)$.
\end{proposition}
\begin{proof}
Firstly it is obvious that $\tilde{G}(x)\leqslant G(x),\ \forall x$.
If $g(x)<-f_1(T,x)$, then by continuity we can find an interval
$[x,x+\Delta x]$ such that for any $y\in[x,x+\Delta x]$,
$g(y)<-f_1(T,y)$. Therefore, $G(x+\Delta x)=G(x)+\int_x^{x+\Delta
x}e^{-cT}g(y)dy<G(x)-\int_x^{x+\Delta x}e^{-cT}f_1(T,y)dy$, which
implies that $\tilde{G}(x)\leqslant G(x+\Delta x)+\int_x^{x+\Delta
x}e^{-cT}f_1(T,y)dy<G(x)$.

Now since $\tilde{G}(x)=\int_x^{x+\Delta
x}e^{-cT}f_1(T,y)dy+\tilde{G}(x+\Delta x)$ for small $\Delta x$, it
can be easily derived that $\tilde{G}'(x)=-e^{-cT}f_1(T,x)$.

The rest of the proposition can be proved in a similar way.
\end{proof}

\begin{proposition}
$\tilde{g}(x)$ is continuous.
\end{proposition}
\begin{proof} Let $I=(-\infty,A)$ and $II=(B,\infty)$, then certainly $g(A)=-f_1(T,A)$, $g(B)=f_2(T,B)$
and $A<B$ by Assumption \ref{ABmod}. We have shown that
$\tilde{G}(x)<G(x)$ and $\tilde{g}(x)=-f_1(T,x)$, $\forall x\in I$,
and $\tilde{G}(x)<G(x)$, $\tilde{g}(x)=f_2(T,x)$, $\forall x\in II$.
Since for any $x\geqslant A$, $g(x)\geqslant -f_1(T,x)$, then it is
easy to see that $\min_{y_1\geqslant
0}[G(x+y_1)+\int_x^{x+y_1}e^{-cT}f_1(T,u)du]=G(x)$, i.e., $y_1=0$.
Similarly for any $x\leqslant B$, $g(x)\leqslant f_2(T,x)$, so
$\min_{y_2\geqslant
0}[G(x-y_2)+\int_{x-y_2}^{x}e^{-cT}f_2(T,u)du]=G(x)$, i.e., $y_2=0$.
As a conclusion, on $[A,B]$ we have $\tilde{G}(x)=G(x)$ and
$\tilde{g}(x)=g(x)$, hence the continuity of $\tilde{g}(x)$.
\end{proof}

In fact we can tell that $\tilde{g}(x)$ satisfies Assumption
\ref{AB}, and in particular, \begin{displaymath}\tilde{g}(x)\in
H^1(\min\{a(T),A\},\max\{b(T),B\}).\end{displaymath}

Now we are ready to consider the newly modified zero-sum game
\begin{eqnarray}\label{jcostnew}
\tilde{J}_z(\tau,\sigma)=\tilde{J}_{(s,x)}(\tau,\sigma)&=&E_{(s,x)}\left\{\int_s^{\tau\wedge\sigma\wedge T}e^{ct}h(t,Y_{t})dt\right.\nonumber\\
&&+1_{(\sigma <\tau\wedge T)}(-f_1(\sigma,Y_{\sigma}))+1_{(\tau <\sigma\wedge T)}f_2(\tau,Y_{\tau})\nonumber\\
&&\left.+1_{(T \leqslant\tau\wedge
\sigma)}\tilde{g}(Y_{T})\right\},\ \ \tau\wedge\sigma\geqslant s,
\end{eqnarray} where  $Y_t$ follows
\begin{equation}\label{yprocnew}
dY_t=(\sigma(Y_t)\sigma'(Y_t)+\mu(Y_t))dt+\sigma(Y_t)dB_t, \ \
Y_s=x.
\end{equation}
The value of this game is thus given by
\begin{equation}\label{vsupinf}
V(z)=V(s,x)=\inf_\tau\sup_\sigma
\tilde{J}_z(\tau,\sigma)=\sup_\sigma \inf_\tau
\tilde{J}_z(\tau,\sigma),\ \ Y_s=x.
\end{equation}

The following is a direct result from Section \ref{tiDynkin}. Notice
that the two free boundary curves $\tilde{a},\tilde{b}$ satisfy
$\tilde{a}(T)=\min\{a(T),A\}$ and $\tilde{b}(T)=\max\{b(T),B\}$.

\begin{theorem}\label{Vregnew} Assume Assumptions \ref{assumab},
\ref{tildeAB}, \ref{monoF1} \ref{monoF2} and \ref{ABmod}. If
$\tilde{a}(t),\tilde{b}(t)$ are $C^1$ functions of $t$, then the
function $V$ in (\ref{vsupinf}) is bounded and continuous and is the
unique weak solution of the following problem:
\begin{eqnarray}
&&\mathcal{L}V(t,y)+e^{ct}h(t,y)=0,\ \ \ \ \ \ \ \ (t,y)\in D\tilde{a}\tilde{b},  \label{Vpdenew}\\
&&\mathcal{L}V(t,y)+e^{ct}h(t,y)<0,\ \ \ \ \ \ \ \ (t,y)\in D\tilde{a},\nonumber\\
&&\mathcal{L}V(t,y)+e^{ct}h(t,y)>0,\ \ \ \ \ \ \ \ (t,y)\in D\tilde{b},\nonumber\\
&&-f_1(t,y)<V(t,y)<f_2(t,y), \ \  (t,y)\in D\tilde{a}\tilde{b},\nonumber\\
&&V(t,y)=-f_1(t,y),\ \ \ \ (t,y)\in D\tilde{a},\nonumber\\
&&V(t,y)=f_2(t,y),\ \ \ \ \ (t,y)\in D\tilde{b},\nonumber\\
&&V(T,y)=\tilde{g}(y).\nonumber
\end{eqnarray}
By considering $V(t,y)$ as a mapping $V:[0,T]\to H^2({\bf U})$,
where ${\bf U}$ is the open interval $(\tilde{A},\tilde{B})$, then
$V\in L^2(0,T:H^2({\bf U}))\cap L^\infty(0,T;H^1({\bf U}))$,
$\frac{d V}{d t}\in L^2(0,T;L^2({\bf U}))$.

Furthermore, $V(t,y)$ is in $C^{0,1}([0,T)\times {\bf U})\cap
C([0,T]\times{\bf U})$.
\end{theorem}

For the related singular control problem, if we define
\begin{displaymath}\begin{aligned}
&W(s,x)=\int_0^x e^{-cs}V(s,y)dy, \\
&H(s,x)=\int_0^xh(s,y)dy + C(s), \\
&G(x)=\int_0^x e^{-cT}g(y)dy,
\end{aligned}
\end{displaymath} where
\begin{displaymath}
C(s)=-\frac{1}{2}\sigma^2(\tilde{a}(s))\frac{\partial^2 W}{\partial
x^2}(s,\tilde{a}(s))-\mu(\tilde{a}(s))\frac{\partial W}{\partial
x}(s,\tilde{a}(s))-\frac{\partial W}{\partial
s}(s,\tilde{a}(s))-\int_0^{\tilde{a}(s)}h(s,y)dy,
\end{displaymath} and define $\tilde{G}(x)$ as in (\ref{G-tilde}), then we immediately get the
following result from Section \ref{tiSingular}

\begin{theorem}\label{verinew}
Assume Assumptions \ref{assumab}, \ref{tildeAB}, \ref{monoF1}
\ref{monoF2} and \ref{ABmod}. If $\tilde{a}(s),\tilde{b}(s)$ are
$C^1$ functions of $s$, then there exists a $C^{1,2}([0,T)\times
{\bf U})\cap C^{0,1}([0,T]\times {\bf U})$ function $W(s,x)$ which
satisfies
\begin{eqnarray*}
\frac{1}{2}\sigma^2(x)\frac{\partial^2 W}{\partial
x^2}(s,x)+\mu(x)\frac{\partial W}{\partial x}(s,x)+\frac{\partial
W}{\partial s}(s,x)+H(s,x)=0,\ \
x\in(\tilde{a}(s),\tilde{b}(s)),&&\\
\frac{1}{2}\sigma^2(x)\frac{\partial^2 W}{\partial
x^2}(s,x)+\mu(x)\frac{\partial W}{\partial x}(s,x)+\frac{\partial
W}{\partial s}(s,x)+H(s,x)>0,\ \ x\in
[\tilde{A},\tilde{a}(s))\cup(\tilde{b}(s),\tilde{B}].&&\\
\frac{\partial W}{\partial x}(s,x)=-e^{-cs}f_1(s,x),\
x\leqslant\tilde{a}(s),\ \ \ \frac{\partial W}{\partial
x}(s,x)=e^{-cs}f_2(s,x),\ x\geqslant\tilde{b}(s), &&\\
-e^{-cs}f_1(s,x)<\frac{\partial W}{\partial
x}(s,x)<e^{-cs}f_2(s,x),\ \
x\in(\tilde{a}(s),\tilde{b}(s)),&&\\
W(T,x)=\tilde{G}(x),\ \ x\in {\bf U}.&&
\end{eqnarray*}

Let $k_{\mathcal S}(z)=k_{\mathcal S}(s,x)$ be given by
\begin{eqnarray}\label{kcostjumpnew} k_{\mathcal{S}}(z)=&&E_{(s,x)}\left(\int_s^T H(t, X_t)dt+G(X_T)\right)\\
&&+E_{(s,x)}\left(\int_s^T e^{-ct}\left(f_1(t,
X_t)dA_t^{(1),c}+f_2(t, X_t)dA_t^{(2),c}\right)\right)\nonumber\\
&&+E_{(s,x)}\left(\sum_{s\leqslant t\leqslant
T}e^{-ct}\left(\int_{X_{nt^-}}^{X_{nt^-}+\Delta A_t^{(1)}} f_1(t,
y)dy\right.\right.\nonumber\\
&&\left.\left.+\int_{X_{nt^-}-\Delta A_t^{(2)}}^{X_{nt^-}} f_2(t,
y)dy\right)\right),\nonumber
\end{eqnarray}
then
\begin{enumerate}

\item For any admissible policy $\mathcal{S}$, $W(s, x)\leqslant k_{\mathcal
S}(s, x),\ \forall (s,x)\in [0,T]\times\mathbb{R}$.

\item $W(s, x)= k_{{\mathcal S}^*}(s,
x),\ \forall (s, x)\in[0,T]\times\mathbb{R}$, where $\mathcal{S}^*$
is the control policy to reflect the process $X_t$ within the region
$D\tilde{a}\tilde{b}$.
\end{enumerate}
\end{theorem}
\begin{proof}
We only need to verify that $W$ as given is optimal. Following the
proof of Theorem \ref{veri} we arrive at
\begin{eqnarray}\label{kwdiffnew}
&&k_{\mathcal{S}}(s,x)-W(s,x)\nonumber\\
&&=E_{(s,x)}\left(k_{\mathcal{S}}(\tau,X_\tau)-W(\tau,X_\tau)\right)\nonumber\\
&& +E_{(s,x)}\int_s^\tau\left( H(t,X_t)+\frac{\partial W}{\partial
t}(t,X_t)+\mu(X_t)\frac{\partial W}{\partial
x}(t,X_t)+\frac{1}{2}\sigma^2(X_t)\frac{\partial^2 W}{\partial
x^2}(t,X_t)\right)dt \nonumber\\
&&+ E_{(s,x)}\int_s^\tau \left(e^{-ct}f_1(t,X_t)+\frac{\partial W}{\partial x}(t,X_t)\right)dA_t^{(1),c}\nonumber\\
&&+E_{(s,x)}\int_s^\tau\left(e^{-ct}f_2(t,X_t)-\frac{\partial
W}{\partial
x}(t,X_t)\right)dA_t^{(2),c}\nonumber\\
&&+E_{(s,x)}\sum_{s\leqslant t< \tau}\Delta
W(t,X_t)\nonumber\\
&&+ E_{(s,x)}\left(\sum_{s\leqslant
t<\tau}\left(\int_{X_{nt^-}}^{X_{nt^-}+\Delta A_t^{(1)}}
e^{-ct}f_1(t, y)dy+\int_{X_{nt^-}-\Delta A_t^{(2)}}^{X_{nt^-}}
e^{-ct}f_2(t, y)dy\right)\right).
\end{eqnarray} By Theorem \ref{wmain}, the second, third and fourth
expectations in (\ref{kwdiffnew}) are all nonnegative, and this
holds true as $\tau\to T$.

Define the sets
\begin{displaymath}\Gamma_+=\{t\in[s,T]: \Delta
A_t^{(1)}>0\},\quad \Gamma_-=\{t\in[s,T]: \Delta A_t^{(2)}>0\},
\end{displaymath} then once again $\Gamma_+\cap\Gamma_-=\phi$.

If we send $\tau\to T$, $k_{\mathcal{S}}(T^-,X_{T^-})$ can be
written as $G(X_T)$ plus the possible jump of control at $T$. Thus
by sending $\tau\to T$ and the fact that
$k_{\mathcal{S}}(T,X_T)=G(X_T)$,
$W(T^-,X_{T^-})=W(T,X_{T^-})=\tilde{G}(X_{T^-})$ by the continuity
of $W$, we can rewrite the remaining parts in (\ref{kwdiffnew}) as
\begin{eqnarray}\label{rewrKWnew}
&&E_{(s,x)}\left(G(X_T)-\tilde{G}(X_{T^-})\right)+E_{(s,x)}\sum_{s\leqslant
t< T}\Delta W(t,X_t) \nonumber\\
&&+ E_{(s,x)}\left(\sum_{s\leqslant t\leqslant
T}\left(\int_{X_{nt^-}}^{X_{nt^-}+\Delta A_t^{(1)}} e^{-ct}f_1(t,
y)dy+\int_{X_{nt^-}-\Delta A_t^{(2)}}^{X_{nt^-}}
e^{-ct}f_2(t, y)dy\right)\right)\nonumber\\
&&=E_{(s,x)}\left(G(X_T)-\tilde{G}(X_{T^-})\right)\nonumber\\
&&+E_{(s,x)}\left(\int_{X_{T^-}}^{X_{T^-}+\Delta A_T^{(1)}}
e^{-ct}f_1(t, y)dy+\int_{X_{T^-}-\Delta A_T^{(2)}}^{X_{T^-}}
e^{-ct}f_2(t, y)dy\right)\nonumber\\
&&+E_{(s,x)}\left[\sum_{s\leqslant t<
T}\int_{X_{nt^-}}^{X_{nt^-}+\Delta A_t^{(1)}}\left(\frac{\partial
W}{\partial x}(t,y)+ e^{-ct}f_1(t,
y)\right)dy\right]\nonumber\\
&&+E_{(s,x)}\left[\sum_{s\leqslant t <T}\int_{X_{nt^-}-\Delta
A_t^{(2)}}^{X_{nt^-}} \left(-\frac{\partial W}{\partial
x}(t,y)+e^{-ct}f_2(t, y)\right)dy\right]
\end{eqnarray} As similar to the proof of Theorem \ref{veri}, the last two
integrals are nonnegative. By the definition of $\tilde{G}$ in
(\ref{G-tilde}), we can write
\begin{displaymath}
\tilde{G}(X_{T^-})=\min_{\Delta A_T^{(1)},\Delta
A_T^{(2)}}\left(G(X_T)+\int_{X_{T^-}}^{X_{T^-}+\Delta A_T^{(1)}}
e^{-ct}f_1(t, y)dy+\int_{X_{T^-}-\Delta A_T^{(2)}}^{X_{T^-}}
e^{-ct}f_2(t, y)dy\right),
\end{displaymath} and now it can be seen that the quantity
(\ref{rewrKWnew}) is nonnegative, hence
$k_{\mathcal{S}}(s,x)\geqslant W(s,x)$.

The rest of the proof is similar to that of Theorem \ref{veri}.
\end{proof}

\section{Regularity of the Free Boundaries}\label{reg_B}
It should be noticed that the $C^1$ regularity of the two free
boundaries $\tilde{a}(t)$ and $\tilde{b}(t)$ are crucial in showing
the regularity of the value function of the Dynkin game, see, e.g.,
the proofs of Theorem \ref{Vreg} and Lemma \ref{bcond5}. In
\cite{Dai09}, the authors claimed that one free boundary is
continuous, and the other is $C^\infty$ by the same arguments as in
Friedman \cite{Friedman75}. However, in \cite{Friedman75} only the
$C^1$ regularity was proved, under some conditions. Karatzas studied
a particular one dimensional singular stochastic control problem in
\cite{Karatzas83}, also through a game approach, and he claimed that
the free boundary is continuously differentiable on $[\epsilon,T]$,
for any $0<\epsilon<T$. However, for a zero-sum game of a general
diffusion process with general obstacles and general terminal
payoff, the regularity of the value function and the regularity of
the free boundaries is still not fully understood. The closely
related problem is the one-sided optimal stopping problem such as
the American option problem, as discussed in \cite{Peskir05}, where
the author characterized the free boundary as the unique solution to
a free-boundary integral equation. Further result can be found in
\cite{Chen07}. Later Yang, Jiang and Bian \cite{Yang09} proved that
the free boundary is continuously differentiable under a particular
condition, which was further relaxed by Bayraktar and Xing
\cite{Bayraktar09}. But in that work, ``...it is essential to have
the value function $V(S,t)$ as the classical solution of the free
boundary problem''. The point is,
 it is often hard to tell which result comes first. If one is
familiar with closed-form solutions of PDE with free boundaries,
once the solution to the PDE is found, the free boundaries are
obtained simultaneously, and vice versa. A very recent result on the
regularity of the free boundary for the American put option can be
found in \cite{Chen12} where the author proved the $C^\infty$
regularity of the free boundary for American put options with
dividend payment. This work may shed some light on the regularity of
the free boundaries of two-obstacle problem with general terminal
payoff function, and the author shall leave this problem as an
interesting future research topic.

\section*{Concluding Remarks} In this paper, we studied a type of time
inhomogeneous stochastic singular control problems of one
dimensional diffusion. We first investigated the optimal policy and
the regularity of the value function $V$ of a time inhomogeneous
zero-sum game (Dynkin game), then the integrated form of this value
function, which is $C^{1,2}$, coincides with the optimal value
function $W$ of the singular control problem. Thus the existence of
a classical solution to the HJB equation is proved. We also
characterized the optimal control policy as to reflect the diffusion
between two time inhomogeneous curves, which are the free boundaries
of the HJB equation. It can be concluded that the $C^1$ property of
the free boundaries is critical for obtaining the smoothness of the
solution of the HJB equation.

It is well known that the time inhomogeneous stochastic singular
control problem and zero-sum game have numerous variations, and very
often the difficulty arises in finding the optimal control policies
and the regularity of the value functions.  This paper aims to set a
basis for further searches of the form of optimal control policies,
as well as the existence and regularity of the solutions of the HJB
equations that are associated with more general time inhomogeneous
singular control problems.

\end{document}